\documentclass[11pt,a4paper,reqno]{amsart}
\usepackage[utf8]{inputenc}
\usepackage[T1]{fontenc}

\usepackage{CJK}

\usepackage{lmodern}
\usepackage{microtype}
\usepackage{mathrsfs}
\usepackage{amsmath}

\usepackage[foot]{amsaddr}

\usepackage{amsfonts}
\usepackage{geometry}
\usepackage{amsmath,amsthm,amssymb,amscd}
\usepackage{latexsym}
\usepackage[final]{hyperref}
\usepackage[abbrev,backrefs]{amsrefs}

\allowdisplaybreaks
\numberwithin{equation}{section}

\newtheorem{theorem}{Theorem}[section]

\newtheorem{lemma}[theorem]{Lemma}
\newtheorem{proposition}[theorem]{Proposition}
\theoremstyle{definition}

\newtheorem{remark}[theorem]{Remark}

\numberwithin{equation}{section}

\newcommand{\abs}[1]{\lvert#1\rvert}
\newcommand{\Bigabs}[1]{\Bigl\lvert#1\Bigr\rvert}
\newcommand{\st}{\;\vert\;}

\newcommand{\R}{\mathbb{R}^N}
\newcommand{\dif}{\,\mathrm{d}}

\title[Local nonlinear perturbation of lower critical Choquard equations]
      {Groundstates for a local nonlinear perturbation of the Choquard equations with lower critical exponent}

\author[J. Van Schaftingen]{Jean Van Schaftingen$^1$}
\address{$^1$Institut de Recherche en Math\'{e}matique et Physique\\
Universit\'{e} catholique de Louvain\\ Chemin du Cyclotron 2 bte L7.01.01\\ 1348 Louvain-la-Neuve, Belgium}
\email{Jean.VanSchaftingen@uclouvain.be}

\author[J. Xia]{Jiankang Xia (夏健康)$^{2}$}
\address{$^{2}$Department of Applied Mathematics \\
Northwestern Polytechnical University, Xi'an 710129, China}
\email{jiankangxia@nwpu.edu.cn}

\thanks{J.\thinspace Van Schaftingen is supported by the Projet de Recherche (Fonds de la Recherche Scientifique--FNRS) T.1110.14 ``Existence and asymptotic behavior of solutions to systems of semilinear elliptic partial differential equations''.
J. Xia is partially supported by NSF of China (NSFC-11771324) and by the Fundamental Research Funds for the Central Universities (in Northwestern Polytechnical University).}

\begin{document}

\begin{abstract}
We prove the existence of ground state solutions by variational methods
to the nonlinear Choquard equations with a nonlinear perturbation
\[
-{\Delta}u+ u=\big(I_\alpha*|u|^{\frac{\alpha}{N}+1}\big)|u|^{\frac{\alpha}{N}-1}u+f(x,u)\qquad \text{ in } \mathbb{R}^N
\]
where $N\geq 1$, $I_\alpha$ is the Riesz potential of order $\alpha \in (0, N)$,
the exponent $\frac{\alpha}{N}+1$ is critical with respect to the Hardy--Littlewood--Sobolev inequality
and the nonlinear perturbation $f$ satisfies suitable growth and structural assumptions.
\end{abstract}

\begin{CJK}{UTF8}{gbsn}
\maketitle 
\end{CJK}

\section{Introduction and main results}
We are interested in the \emph{nonlinear Choquard equation}
\begin{equation}
\label{eqequation}
\left\{
\begin{aligned}
&-{\Delta}u+u=\big(I_\alpha*|u|^{\frac{\alpha}{N}+1}\big)|u|^{\frac{\alpha}{N}-1}u+ f(x,u) \qquad \text{ in } \R\\
&u\in H^1{(\R})
\end{aligned} \right.
\end{equation}
where $N\geq 1$, $I_{\alpha}:\R\to\mathbb{R}$ is a \emph{Riesz potential} of order  $\alpha\in(0,N)$ defined at each point $x\in\mathbb{R}^N\backslash{\{0\}}$ by
\[
  I_{\alpha}(x)=\frac{A_\alpha}{|x|^{N-\alpha}} \;\text{ with }\;
  A_\alpha=\frac{\Gamma(\frac{N-\alpha}{2})}{2^\alpha\pi^{\frac{N}{2}}\Gamma(\frac{\alpha}{2})},
\]
with $\Gamma$ denoting the classical Gamma function and $*$ the convolution product on the Euclidean space \(\R\) and \(f : \R \times \mathbb R \to \mathbb R\) is a nonlinear perturbation.

In the case \(f = 0\), the Choquard equation \eqref{eqequation} reduces to the well-known Choquard--Peark equation in $\mathbb{R}^N$
 \begin{equation*}
 \tag{$\mathcal{C}$}
 \label{eqc}
 -{\Delta}u+ V u=\big(I_{\alpha}*|u|^{p}\big)|u|^{p-2}u  \qquad \text{ in } \mathbb{R}^N,
\end{equation*}
When \(N = 3\), $\alpha=2$, $p=2$ and $V$ is a positive constant, this equation appears in several physical contexts, such as standing waves for the \emph{Hartree equation}, the description by Pekar of the quantum physics of a \emph{polaron at rest} \cite{P}, the model by Choquard of an \emph{electron trapped in its own hole} \cite{L} or the coupling of the Schr\"odinger equation under a classical \emph{Newtonian gravitational potential}  \cite{Diosi1984,J1,J2,MPT}.

Existence and qualitative properties of the Choquard equation \eqref{eqc} have been studied for a few decades by variational methods. In a pioneering work, E. H. Lieb first obtained the existence and uniqueness of positive solutions to Choquard's equation \eqref{eqc} in $\mathbb{R}^3$ with $V=1$, $\alpha=2$ and $p=2$  \cite{L}. Later, P. -L. Lions \cite{Lions1980,Lions1984} got the existence and multiplicity results of normalized solution on the same topic. When the potential $V$ is a positive constant, V. Moroz and the first author established the existence of ground state solutions to the Choquard equation \eqref{eqc} in \cite{MVJFA} within an  optimal range of $p$ that $p$ satisfies the intercriticality condition
$$
 \frac{N-2}{N+\alpha}<\frac{1}{p}<\frac{N}{N+\alpha}.
$$
They investigated extensively the qualitative properties of solutions to the Choquard equation \eqref{eqc} such as the regularity, positivity and radial symmetry decay behavior at infinity. For more related topics, we refer the reader to the recent survey paper \cite{MVSReview}.

In view of the Poho\v{z}aev identity \cite{Menzala1983,CingolaniSecchiSquassina2010,GenevVenkov2012,MVJFA,MVCCM,MVTAMS}, the Choquard equation \eqref{eqc} with $V=1$ has no nontrivial smooth $H^1$ solution when either $p\leq \frac{\alpha}{N}+1$ or $p\geq \frac{N+\alpha}{N-2}$. The endpoints of the above interval  are \emph{critical exponents} that come from the Hardy--Littlewood--Sobolev inequality (see Proposition~\ref{propHLS}). The \emph{upper critical exponent} plays a similar role as the Sobolev critical exponent in the local semilinear equations \cite{GY,BN}. The \emph{lower critical exponent} $\frac{\alpha}{N}+1$ seems to be a new feature for Choquard's equation, which is related to a new phenomenon of ``bubbling at infinity'' \cite{MVCCM}, V. Moroz and the first author there established the existence of  ground state solution under the assumption that \(V\) is asymptotically close enough to its limit at infinity (see also \cite{CVZ}). When the potential \(V\) is coercive, ground state solutions of \eqref{eqc} exist for the lower critical exponent \cite{VanSchaftingenXia}.

\medbreak

In the present work, we examine how the presence of a nonlinear perturbation $f$, instead of the linear perturbation in \cite{MVCCM} influences the situation. As the first model, we study the following autonomous nonlinear Choquard equation
 \begin{equation*}
 \tag{$\mathcal{C}_*$}
 \label{eqcs}
\left\{
\begin{aligned}
&-{\Delta}u+u=\big(I_\alpha*|u|^{\frac{\alpha}{N}+1}\big)|u|^{\frac{\alpha}{N}-1}u+ f(u) \qquad \text{ in } \R\\
&u\in H^1{(\R})
\end{aligned} \right.
\end{equation*}
We show that this equation admits a ground state solution if the decay rate of the perturbation $f$ near \(0\)
is not too fast:
\begin{theorem}
\label{thm1.1}
For every $N\geq 1$ and $\alpha\in(0,N)$, there exists $\Lambda_0>0$ such that if the function $f\in C(\mathbb{R},\mathbb{R})$ satisfies
\begin{itemize}
  \item [$(f_1)$] $f(t)=o(t)$ as $t\rightarrow 0$,
  \item [$(f_2)$]  $\abs{f(t)}\leq a(\abs{t} + \abs{t}^{q-1})$ for some $a > 0$ and $q>2$ with $\frac{1}{q}>\frac{1}{2}-\frac{1}{N}$,
  \item [$(f_3)$] there exists $\mu > 2$ such that
    $$0<\mu F(t)\leq f(t)t\text{ for all } t\neq 0$$
where $F(t)=\int^{t}_{0}f(s)\dif s$,
  \item [$(f_4)$] $\varliminf_{|t|\to 0}\frac{F(t)}{t^{\frac{4}{N}+2}}\geq \Lambda_0$,
     \end{itemize}
then the Choquard equation \eqref{eqcs} has a ground state solution.
\end{theorem}

The solution $u$ obtained in Theorem~\ref{thm1.1} is a ground state solution in the sense that it minimizes the corresponding variational functional $\mathcal{J}$, see \eqref{eqenergyfunctional} below, among nontrivial solutions, namely, the solution \(u\) has the least energy among nontrivial solutions.

A natural way to search for ground state solutions is to minimize the corresponding functional on the so called Nehari manifold, which is of use especially when the nonlinearity admits some suitable monotonic properties, see the survey paper \cite{SW} for details. Unfortunately, we do not have such a monotonicity assumption in our setting, which makes the Nehari manifold method unsuitable.
The main idea in our proof is first to show that the functional $\mathcal{J}$ has a nontrivial critical point by the mountain-pass lemma and a concentration compactness argument, and then to look for a minimizer for the following minimization problem
\begin{equation}
\label{minimizationproblem}
m_0:=\inf\, \bigl\{\mathcal{J}(u)\,|\,u\in H^1(\R)\setminus\{0\} \text{ and } \mathcal{J}'(u)=0\bigr\}.
\end{equation}
The minimizer is then a ground state solution of nonlinear Choquard equation \eqref{eqcs}.

In the case $N\geq 2$, under the additional assumptions that the function $f$ is odd and has constant sign on $(0,+\infty)$, we obtain a radially symmetric solution to \eqref{eqcs} by the Schwarz symmetrization \cite{JeanVSCCM} and the symmetric variational principle \cite{JeanVSCCM}*{Theorem 3.2}, and furthermore, following an argument of L. Jeanjean and H. Tanaka \cite{Jeanjeanremark}, this radial solution is a ground state.

\begin{theorem}\label{thmsymmetricsolution}
Let $N\geq 2$ and $\alpha\in(0,N)$. If the function $f\in C(\mathbb{R},\mathbb{R})$ satisfies the conditions $(f_1)-(f_4)$ in Theorem \ref{thm1.1}, $f$ is odd and \(f\) has constant sign on $(0,+\infty)$, then the Choquard equation \eqref{eqcs} admits a  groundstate which is a radial function.
\end{theorem}

Our final result is also on the existence of ground state solutions to problem \eqref{eqequation} with a non-autonomous homogeneous perturbation,
 \begin{equation*}
 \tag{$\mathcal{C}_K$}
 \label{eqck}
\left\{
\begin{aligned}
&-{\Delta}u+u=\big(I_\alpha*|u|^{\frac{\alpha}{N}+1}\big)|u|^{\frac{\alpha}{N}-1}u+ K(x)|u|^{q-2}u \qquad \text{ in } \R\\
&u\in H^1{(\R})
\end{aligned} \right.
\end{equation*}
where $K\in L^\infty(\R)$ is a positive anti-potential well. The appearance of the potential $K$ breaks down the invariance under translations in the Euclidean space \(\R\) and brings up different challenges.

\begin{theorem}\label{thm1.3}
Let $N\geq 1$, $q\in(2,2+\frac{4}{N})$ and $K\in L^{\infty}(\R)$. If
\[
  \inf_{x\in\R}K(x)= K_\infty=\lim_{|x|\to\infty}K(x)>0,
\]
then, \eqref{eqck} admits a ground state solution.
\end{theorem}

As mentioned above, we cannot use the translation-invariant concentration-compact\-ness argument directly due to the appearance of potential $K$.
Our proofs borrow some ideas from proofs of the existence of ground state solutions on some kinds of local semi-linear problems in $\R$ \cite{SW,WangXia}. We follow a similar strategy as Theorem \ref{thm1.1} but depend on a comparison of the energy with the corresponding limit problem, which turns out to be \eqref{eqcs} with a homogeneous perturbation $f(t)=K_\infty|t|^{q-2}t$, see \eqref{eqcinfty} in Section~\ref{finalsection} below. More precisely, under our assumptions on the potential $K$, the fact that the ground state energy of \eqref{eqck} is strictly less than that of the limit problem \eqref{eqcinfty} plays an important role and ultimately restores the compactness.

The rest of the paper is organized as follows. We give some preliminaries and a key estimate on the mountain pass energy level in Section~\ref{sectionPreliminary}. Theorem~\ref{thm1.1} is proved in Section~\ref{sectionProofs}. Section \ref{sectionsymmetricvarprin} is devoted to the Schwarz symmetrization arguments, which provides an alternative to concentration-compactness arguments thus completes the proof of Theorem \ref{thmsymmetricsolution}. The proof of Theorem~\ref{thm1.3} is given in the final section~\ref{finalsection}.

\section{Preliminaries and energy estimates}
\label{sectionPreliminary}
Our functional analytic framework is the the classical Sobolev space $H^1(\R)$ equipped with the standard norm $\|\cdot\|$,
$$
\|u\|^2=\int_{\R}\abs{\nabla u}^2+\abs{u}^2.
$$
The Choquard equation \eqref{eqcs} is variational in nature, the corresponding functional $\mathcal{J}:H^1(\R)\to \mathbb{R}$ is defined for every function \(u \in H^1 (\R)\) by
\begin{equation}
\label{eqenergyfunctional}
\mathcal{J}(u) =\frac{1}{2}\int_{\R}\abs{\nabla u}^2+\abs{u}^2-\frac{N}{2(N+ \alpha)}\int_{\R}(I_\alpha*|u|^{\frac{\alpha}{N}+1})|u|^{\frac{\alpha}{N}+1}-\int_{\R}F(u).
\end{equation}

The following classical Hardy--Littlewood--Sobolev inequality is a starting point of the variational approach to the problem \eqref{eqcs} and implies by standard arguments \cite{MVJFA,MVTAMS}  the well-definiteness, continuity and differentiability of the nonlocal term in the functional $\mathcal{J}$ defined by \eqref{eqenergyfunctional}.
\begin{proposition}
[\cite{LL}*{Theorem 4.3}]
\label{propHLS}
Let $N\geq 1$, $\alpha\in(0, N)$ and $s\in (1, \frac{N}{\alpha})$. Then  for any  $\varphi\in L^s (\R)$,
$I_\alpha*\varphi\in L^{\frac{Ns}{N-\alpha s}}(\R)$, and
\begin{equation}
\label{hls}
\int_{\mathbb{R}^N}|I_\alpha*\varphi|^{\frac{Ns}{N-\alpha s}}\leq C \Big(\int_{\R}|\varphi|^s\Big)^{\frac{N}{N-\alpha s}}
\end{equation}
where the constant $C>0$ depends only on  $\alpha$, $N$ and $s$.
\end{proposition}

By the semi-group identity for the Riesz potential $I_\alpha=I_{\alpha/{2}}*I_{\alpha/{2}}$ \cite{LL}*{Corollary 5.10},  the Hardy--Littlewood--Sobolev inequality \eqref{hls} can be rewritten as
\begin{equation*}
\int_{\R} (I_\alpha*\abs{u}^{p})\abs{u}^{p}=\int_{\R}\big|I_{\alpha/2}*|u|^{p}\big|^2\leq C\Big(\int_{\R}\abs{u}^\frac{2Np}{N+\alpha}\Big)^{{\frac{\alpha}{N}+1}}.
\end{equation*}
The exponent $\frac{\alpha}{N}+1$ is critical, since the functional $\mathcal{J}$ is well defined in $H^1(\R)$ if and only if
$$\frac{\alpha}{N}+1\leq p\leq \frac{N+\alpha}{(N-2)_+}.$$
In our setting, the Hardy--Littlewood--Sobolev inequality \eqref{hls} turns out to be
\begin{equation}
\label{eqhls}
\int_{\R} (I_\alpha*\abs{u}^{\frac{\alpha}{N}+1})\abs{u}^{\frac{\alpha}{N}+1}=\int_{\R}\big|I_{\alpha/2}*|u|^{\frac{\alpha}{N}+1}\big|^2\leq C_H\Big(\int_{\R}\abs{u}^2\Big)^{{\frac{\alpha}{N}+1}}
\end{equation}
where the constant $C_H>0$ depends only on the dimension $N$ and on the order $\alpha$.
It can be restated in terms of minimizers of the following minimization problem.
\begin{equation}
\label{minimizertoHLS}
\mathcal{S}=\inf\Big\{\int_{\R}|u|^2 \st u\in H^1(\R) \mbox{ and } \int_{\R}(I_\alpha*|u|^{\frac{\alpha}{N}+1})|u|^{\frac{N+\alpha}{N}}=1\Big\}.
\end{equation}
By \cite{LL}*{Theorem 4.3}, the infimum $\mathcal{S}$ is achieved if and only if for every \(x \in \R\)
\begin{equation}
\label{minimizerofHLS}
 u (x) =A\,\biggl(\frac{\varepsilon}{\varepsilon^2+|x-a|^2}\biggr)^{\frac{N}{2}},
\end{equation}
for some given constants \(A \in \mathbb{R}\), $a\in\R$ and $\varepsilon\in(0,+\infty)$.
The form of minimizers in \eqref{minimizerofHLS} suggests that a loss of compactness in \eqref{eqc}
with $p=\frac{\alpha}{N}+1$ may occur by both of translations and dilations.

In our subsequent arguments we will use the following variant  of the classical Brezis--Lieb lemma for Riesz potentials.
\begin{lemma}
[\cite{MVJFA}*{Lemma 2.4}]
\label{Brezislieb}
Let $N\geq 1$ and $\alpha\in(0,N)$. If the sequence $(u_n)_{n\in\mathbb{N}}$ is a bounded sequence in $L^2(\R)$ and converges to $u$ almost everywhere in $\R$, then
\begin{multline*}
\lim_{n\to\infty}\int_{\R}(I_\alpha*|u_n|^{\frac{\alpha}{N}+1})\,|u_n|^{\frac{\alpha}{N}+1}=
\lim_{n\to\infty}\int_{\R}(I_\alpha*|u_n-u|^{\frac{\alpha}{N}+1})\,|u_n-u|^{\frac{\alpha}{N}+1}\\
+\int_{\R}(I_\alpha*|u|^{\frac{\alpha}{N}+1})\,|u|^{\frac{\alpha}{N}+1}.
\end{multline*}
\end{lemma}

We also have Brezis--Lieb lemma for the nonlinear local term.

\begin{lemma}
\label{splitinglemma}
Assume that $f\in C(\mathbb{R},\mathbb{R})$ and satisfies $(f_2)$.
If the sequence $(u_n)_{n\in\mathbb{N}}$ is bounded in both $L^2(\R)$ and $L^q(\R)$ and converges to $u$ almost everywhere in $\R$,
then
$$
\lim_{n\to\infty}\int_{\R}|F(u_n)-F(u_n-u)-F(u)|=0.
$$
\end{lemma}

The proof is a variant on the classical proof of the Brezis--Lieb lemma \cite{BrezisLieb1983}.
\begin{proof}[Proof of Lemma~\ref{splitinglemma}]
We first deduce by Fatou's lemma that $u\in L^2(\R)\cap L^q(\R)$.
By our assumption $(f_2)$, we have for each \(t \in \mathbb{R}\)
\begin{equation}\label{eqbounded}
|F(t)|\leq \int_{0}^{t}\abs{f(s)}\dif s\leq a (|t|^2+|t|^q),
\end{equation}
we thus obtain that \(F(u_n),\; F(u_n-u),\; F(u)\in L^1(\R)\).
Since $F\in C^1(\mathbb{R},\mathbb{R})$, by using $(f_2)$ again, we have for each \(n \in \mathbb{N}\),
\begin{equation*}
\label{}
\begin{split}
|F(u_n)-F(u_n-u)|&\leq \int_{0}^1|f(u_n-u+\theta u)u|\dif \theta \\
&\leq C(|u_n-u||u|+|u|^2+|u_n-u|^{q-1}|u|+|u|^q).
\end{split}
\end{equation*}
For fixed $\varepsilon>0$, it follows from Young's inequality for products that there exists $C_\varepsilon>0$ such that, for each \(n \in \mathbb{N}\),
\begin{equation*}
|F(u_n)-F(u_n-u)|\leq \varepsilon (\abs{u_n-u}^2+\abs{u_n-u}^q)+C_\varepsilon( \abs{u}^2+\abs{u}^q),
\end{equation*}
which, together with \eqref{eqbounded}, implies that, for each \(n \in \mathbb{N}\),
\begin{equation*}
|F(u_n)-F(u_n-u)-F(u)|\leq \varepsilon (\abs{u_n-u}^2+\abs{u_n-u}^q)+(a+C_\varepsilon)( \abs{u}^2+\abs{u}^q)
\end{equation*}
and then
\[
g_n^\varepsilon:=\Big(|F(u_n)-F(u_n-u)-F(u)|- \varepsilon (\abs{u_n-u}^2+\abs{u_n-u}^q)\Big)_+\leq (a+C_\varepsilon)( \abs{u}^2+\abs{u}^q).
\]
Since the sequence $(g_n^\varepsilon)_{n \in \mathbb{N}}$ converges to \(0\) almost everywhere in $\R$
we deduce by Lebesgue's dominated convergence theorem that
\[
\lim_{n\to\infty}\int_{\R}g_n^\varepsilon= 0.
\]
It then follows that
\begin{equation*}
\lim_{n\to\infty}\int_{\R}|F(u_n)-F(u_n-u)-F(u)|
\leq \varepsilon \limsup_{n \to \infty}\Big(\int_{\R}\abs{u_n-u}^2+\abs{u_n-u}^q\Big).
\end{equation*}
The conclusion follows then by letting \(\varepsilon \to 0\).
\end{proof}

To obtain a Palais--Smale sequence, we show that the functional $\mathcal{J}$ has the mountain pass geometry.
\begin{proposition}
\label{propMountainpass}
The functional $\mathcal{J}$ has the mountain pass geometry:
\begin{itemize}
  \item [(i)] there exists $\rho>0$ such that $\inf_{\|u\|=\rho}\mathcal{J}(u)>0$;
  \item [(ii)] for any $u\in H^1(\R)\setminus\{0\}$, it holds $\lim_{t\to+\infty}\mathcal{J}(tu)=-\infty$.
\end{itemize}
\end{proposition}
\begin{proof}
We reproduce the proof here although it is standard. By  the assumptions $(f_1)$, $(f_2)$, there exist constants $C_1,C_2>0$ such that for each \(t \in \mathbb{R}\),
\[
|f(t)|\leq \frac{1}{2}|t|+C_1|t|^{q-1}\;\;\text{ and }\;\; |F(t)|\leq \frac{1}{4}|t|^2+C_2|t|^q.
\]
We thus deduce by the Hardy--Littlewood--Sobolev inequality \eqref{eqhls} and the classical Sobolev inequality, that
\begin{equation*}
\begin{split}
\mathcal{J}(u)&=\frac{1}{2}\|u\|^2-\frac{N}{2(N+\alpha)}\int_{\R}(I_\alpha*\abs{u}^{\frac{\alpha}{N}+1})\abs{u}^{\frac{\alpha}{N}+1}-\int_{\R}F(u)\\
&\geq \frac{1}{4}\|u\|^2-C\|u\|^{\frac{2\alpha}{N}+2}-C_2\int_{\R}\abs{u}^q\\
&=\|u\|^2\Big(\frac{1}{4}-C\|u\|^{\frac{2\alpha}{N}}-C_3\|u\|^{q-2}\Big).
\end{split}
\end{equation*}
We then have that $\inf_{\|u\|=\rho}\mathcal{J}(u)\geq \frac{1}{8}\rho^2>0$ provided that $\rho$ is sufficiently small.

On the other hand, for any $u\in H^1(\R)\setminus\{0\}$ and \(t \in (0, + \infty)\), we have
\begin{equation*}
\begin{split}
\mathcal{J}(tu)&=\frac{t^2}{2}\|u\|^2-\frac{Nt^{\frac{2\alpha}{N}+2}}{2(N+\alpha)}\int_{\R}(I_\alpha*|u|^{\frac{\alpha}{N}+1})
|u|^{\frac{\alpha}{N}+1}-\int_{\R}F(tu)\\
&\leq \frac{t^2}{2}\|u\|^2-\frac{Nt^{\frac{2\alpha}{N}+2}}{2(N+\alpha)}\int_{\R}(I_\alpha*|u|^{\frac{\alpha}{N}+1})
|u|^{\frac{\alpha}{N}+1},
\end{split}
\end{equation*}
and the conclusion $(ii)$ follows.
\end{proof}
By the classical mountain pass theorem \cite{AmbrosettiRabinowitz1973,Rabinowitz1986,Struwe,W}, we have a min-max description at the energy level $c_0$, defined by
\begin{equation}
\label{eqMountainpassenergy}
c_0=\inf_{\gamma\in \Gamma}\max_{t\in[0,1]}\mathcal{J}(\gamma(t)),
\end{equation}
where
\begin{equation}\label{insectionset}
\Gamma=\big\{\gamma\in C([0,1], H^1(\R))\st \gamma(0)=0,\,\,\,\mathcal{J}(\gamma(1))<0\big\}.
\end{equation}

We finally give an estimate on the mountain pass energy level, which is essential in ensuring compactness.

\begin{lemma}
\label{lemEnergyestimates}
Let $N\geq 1$, \(\alpha \in (0, N)\) and $c_*=\frac{\alpha}{2(N+\alpha)}\mathcal{S}^{\frac{N}{\alpha}+1}$. There exists $\Lambda_0>0$, which only depends on $\alpha$ and \(N\), such that if \((f_4)\) is satisfied then $c_0<c_* $.
\end{lemma}
\begin{proof}
We first show that $c_0\leq c_1$ where
$$
c_1=\inf_{u\in H^1(\R)\setminus\{0\}}\max_{t\geq 0} \mathcal{J}(tu).
$$
Indeed, for any \(u\in H^1(\R)\setminus\{0\}\), by Proposition~\ref{propMountainpass} $(ii)$,
there exists $t_u>0$ such that $\mathcal{J}(t_uu)<0$. Hence, by definition of \(c_0\), we have
\begin{equation}
 \label{ineq_c0_c1}
 c_0\leq \max_{\tau \in[0,1]}\mathcal{J}(\tau t_uu) \leq \max_{t\geq 0}\mathcal{J}(tu),
\end{equation}
which leads to $c_0\leq c_1$, since the left hand side does not depend on the choice of $u$.

According to the representation formula \eqref{minimizerofHLS} for the optimal functions of the Hardy--Littlewood--Sobolev inequality,
we set for \(\varepsilon > 0\) and \(x \in \R\),
$U(x)=A(1+|x|^2)^{-\frac{N}{2}}$ and
$U_\varepsilon(x)=\varepsilon^{\frac{N}{2}}U(\varepsilon x)$.
For each \(\varepsilon > 0\), the function \(U_\varepsilon\) satisfies
\begin{align*}
\int_{\R}|U_\varepsilon|^2 & = \mathcal{S} &
&\text{ and } &
&\int_{\R}(I_\alpha*|U_\varepsilon|^{\frac{\alpha}{N}+1})|U_\varepsilon|^{\frac{\alpha}{N}+1}=1.
\end{align*}
Moreover, through direct computations by changes of variables, we have that
\begin{align*}
\label{eqEsttimates}
\int_{\R}|\nabla U_\varepsilon|^2 &=\varepsilon^{2}\int_{\R}|\nabla U|^2 &
&\text{and }&
\int_{\R}  F(U_\varepsilon(x))\dif x&=\varepsilon^{-N}\int_{\R} F(\varepsilon^{\frac{N}{2}}U).
\end{align*}
For every \(\varepsilon > 0\), we now consider the function \(\xi_\varepsilon : [0, +\infty) \to \mathbb{R}\) defined for each \(t \in [0, + \infty)\) by
\begin{equation*}
\label{eqAuxiliary}
\xi_\varepsilon(t):=\mathcal{J}(tU_\varepsilon)=g(t)+\varphi_\varepsilon(t),
\end{equation*}
where \(g : [0, +\infty) \to \mathbb{R}\) and \(\varphi_\varepsilon : [0, +\infty) \to \mathbb{R}\) are defined for every \(t \in [0, + \infty)\) by
\begin{align*}
g(t)&=\frac{1}{2}\mathcal{S}t^2-\frac{N}{2(N+\alpha)}t^{2+\frac{2\alpha}{N} }&
&\text{ and }&
\varphi_\varepsilon(t)&=\frac{t^2}{2}\int_{\R}|\nabla U_\varepsilon|^2- \int_{\R}F(tU_\varepsilon).
\end{align*}
Since $\xi_\varepsilon(t)>0$ whenever $t>0$ is small enough, $\lim_{t\to 0}\xi_\varepsilon(t)=0$ and $\lim_{t\to+\infty}\xi_\varepsilon(t)=-\infty$, for each \(\varepsilon > 0\)
there exists $t_\varepsilon>0$ such that
$$
\xi_\varepsilon(t_\varepsilon)=\sup_{t\geq 0}\xi_\varepsilon(t)=\max_{t\geq 0}\xi_\varepsilon(t).
$$
By the definition of the function $g$, we have
\begin{equation}
\label{ineq_c1}
c_1\leq \max_{t\geq 0}\xi_\varepsilon(t)=\xi_\varepsilon(t_\varepsilon)= g(t_\varepsilon)+\varphi_\varepsilon(t_\varepsilon)\leq g(t_*)+\varphi_\varepsilon(t_\varepsilon)
\end{equation}
where  $t_*>0$ is unique and satisfies that
$$
g(t_*)=\max_{t\geq 0} g(t)=\frac{\alpha}{2(N+\alpha)}\mathcal{S}^{\frac{N}{\alpha}+1}=c_*.
$$
Since $\xi_\varepsilon'(t_\varepsilon)=0$, we have
\begin{equation}
\label{eq_idxrylp}
\begin{split}
\varepsilon^2 \int_{\R}|\nabla U|^2+\mathcal{S}&=t_\varepsilon^{\frac{2\alpha}{N}}+ \int_{\R}t_{\varepsilon}^{-1}f\big(t_\varepsilon   U_\varepsilon (x)\big)U_\varepsilon (x)\dif x \\
&=t_\varepsilon^{\frac{2\alpha}{N}}+ \int_{\R}\frac{f\big(t_\varepsilon\varepsilon^{\frac{N}{2}}  U(x)\big)}{t_{\varepsilon}\varepsilon^{\frac{N}{2}}U(x)}\abs{U(x)}^2\dif x.
\end{split}
\end{equation}
By $(f_3)$ this implies that
\[
\varepsilon^2 \int_{\R}|\nabla U|^2+\mathcal{S}\geq t_\varepsilon^{\frac{2\alpha}{N}}.
\]
Hence, we have
\(
 \varlimsup_{\varepsilon \to 0} t_\varepsilon^{\frac{2\alpha}{N}} \le \mathcal{S}
\),
which is equivalent to \(
 \varlimsup_{\varepsilon \to 0} t_\varepsilon \le t_*\).
By $(f_1)$, this implies that
\[
\lim_{\varepsilon \to 0}
  \int_{\R}\frac{f\big(t_\varepsilon\varepsilon^{\frac{N}{2}}  U(x)\big)}{t_{\varepsilon}\varepsilon^{\frac{N}{2}}U(x)}\abs{U(x)}^2\dif x = 0,
\]
and thus in view of \eqref{eq_idxrylp}, we have
\(\lim_{\varepsilon \to 0} t_\varepsilon^{\frac{2\alpha}{N}} = \mathcal{S}
\)
and thus \(\lim_{\varepsilon \to 0} t_\varepsilon = t_*\).

We now observe that
\begin{equation*}
\varphi_\varepsilon(t_\varepsilon)
=\frac{t_\varepsilon^2}{2}\varepsilon^2\int_{\R}|\nabla U|^2- \int_{\R}\varepsilon^{-N} F\big(t_\varepsilon \varepsilon^{\frac{N}{2}} U(x)\big) \dif x.
\end{equation*}
By our assumption $(f_4)$, we deduce from Fatou's lemma that
\begin{equation*}
\label{eqEstimatenonlinearity}
\begin{split}
\varliminf_{\varepsilon\to 0}\frac{1}{\varepsilon^{N + 2} t_\varepsilon^{\frac{4}{N} + 2}}\int_{\R} F(t_\varepsilon \varepsilon^{\frac{N}{2}} U(x)) \dif x&=\varliminf_{\varepsilon\to 0} \int_{\R}\frac{ F(t_\varepsilon \varepsilon^{\frac{N}{2}} U)}{{\abs{t_\varepsilon \varepsilon^{\frac{N}{2}}U}}^{\frac{4}{N}+2}}\abs{U}^{\frac{4}{N}+2}  \dif x\\
&\geq  \int_{\R} \Lambda_0  \abs{U}^{\frac{4}{N}+2},
\end{split}
\end{equation*}
we thus obtain that
$$
\varlimsup_{\varepsilon \to 0}\frac{
\varphi_\varepsilon(t_\varepsilon)}{\varepsilon^2} \leq\frac{t_*^2}{2}\int_{\R}\abs{\nabla U}^2 -\Lambda_0  t_*^{\frac{4}{N}+2} \int_{\R}\abs{U}^{\frac{4}{N}+2}.
$$
Hence, there exists $\Lambda_0 > 0\), depending, through \(t_*\) and \(U\), only on \(\alpha\) and \(N\), such that if $\varepsilon > 0$ is small enough $\varphi_\varepsilon(t_\varepsilon)<0$.	It then follows by \eqref{ineq_c1} that \(c_1 < c_*\) and thus \(c_0 < c_*\) in view of \eqref{ineq_c0_c1}.
\end{proof}
\section{Proof of Theorem~\ref{thm1.1}}
\label{sectionProofs}

The proof of Theorem~\ref{thm1.1} will be carried out into two steps. First, we are aiming to find a nontrivial solution of $\eqref{eqcs}$ with its energy level is strictly less than $c_*$, and then we show that the minimization problem \eqref{minimizationproblem} is attained.

Before giving a complete proof, we state the following lemmas,  which will be frequently used in the sequel proofs.
\begin{lemma}\label{lemcompactness}
If $(u_n)_{n\in\mathbb{N}}$ is a sequence in $H^1(\R)$
such that $$\varliminf\limits_{n\to\infty}\|u_n\|>0 \text{ and }\lim\limits_{n\to\infty}\langle \Phi'(u_n),u_n\rangle=0,$$
where the functional $\Phi:H^1(\R)\to \mathbb{R}$ is defined by
$$
  \Phi(u)=\frac{1}{2}\|u\|^2-\frac{N}{2(N+\alpha)}\int_{\R}(I_\alpha*|u|^{\frac{\alpha}{N}+1})|u|^{\frac{\alpha}{N}+1},
$$
then $\varliminf_{n\to\infty}\Phi(u_n)\geq c_*$.
\end{lemma}
\begin{proof}
We observe that, as \(n \to \infty\),
$$
\|u_n\|^2=\int_{\mathbb{R}^N}(I_\alpha*|u_n|^{\frac{\alpha}{N}+1})|u_n|^{\frac{\alpha}{N}+1}+o_n(1),
$$
we deduce by the assumption $\varliminf_{n\to\infty}\|u_n\|>0$ and by the Hardy--Littlewood--Sobolev inequality \eqref{eqhls}, that
$$\varliminf\limits_{n\to\infty}\int_{\R}|u_n|^2>0.$$
It follows from the definition of $\mathcal{S}$ that, as \(n \to \infty\),
\begin{equation*}
\begin{split}
\int_{\R}(I_\alpha*|u_n|^{\frac{ \alpha}{N}+1})|u_n|^{\frac{\alpha}{N}+1}+o_n(1)&\geq\int_{\R}|u_n|^2\\
 &\geq \mathcal{S} \Big(\int_{\R}(I_\alpha*|u_n|^{\frac{\alpha}{N}+1})|u_n|^{\frac{\alpha}{N}+1}\Big)^{\frac{N}{N+\alpha}},
\end{split}
\end{equation*}
which leads to
\begin{equation}\label{equsefulineq}
\varliminf\limits_{n\to\infty}\|u_n\|^2=\varliminf\limits_{n\to\infty}\int_{\R}(I_\alpha*|u_n|^{\frac{\alpha}{N}+1})|u_n|^{\frac{\alpha}{N}+1}\geq \mathcal{S}^{1+\frac{N}{\alpha}}.
\end{equation}
Therefore,
\begin{equation}\label{equsefulineq2}
\begin{split}
\Phi(u_n)&=\Phi(u_n)-\frac{N}{2(N+\alpha)}\langle \Phi'(u_n), u_n\rangle+o_n(1)\\
&= \frac{\alpha}{2(N+\alpha)}
 \|u_n\|^2+o_n(1).
\end{split}
\end{equation}
Then,  the conclusion follows from \eqref{equsefulineq} and \eqref{equsefulineq2} .
\end{proof}

We recall that a sequence $(u_n)_{n\in\mathbb{N}}$ in $H^1(\R)$ is said to be a Palais--Smale sequence  at level $c \in \mathbb{R}$ (short for $(PS)_c$ sequence) of a $C^1$ functional $\Psi: H^1(\R)\to \mathbb{R}$ if it satisfies
\[
\Psi(u_n)\to c \;\;\text{ and }\;\; \Psi'(u_n)\to 0  \text{ in } H^{-1}(\R) \;\,\text{ as } n\to \infty,
\]
where $H^{-1}(\R)$ denotes the dual space of $H^1(\R)$.
\begin{lemma}\label{lemnontrivialsolution}
If $(u_n)_{n\in\mathbb{N}}$ is a bounded $(PS)_{c}$ sequence with $c\in(0,c_*)$ for the functional $\mathcal{J}$, then, up to a subsequence and translations, the sequence \((u_n)_{n \in \mathbb{N}}\) converges weakly to some function \(u \in H^1 (\R)\setminus \{0\}\) such that  $$\mathcal{J}'(u)=0 \text{ and } \mathcal{J}(u)\in(0, c].$$
\end{lemma}
\begin{proof}
First we show that $\varlimsup_{n \to \infty} \int_{\R} \abs{u_n}^q > 0$. Otherwise, up to a subsequence and by combining the assumptions $(f_1)$ and $(f_2)$ we have
\begin{equation}\label{vanishing}
\lim_{n\to\infty}\int_{\R}f(u_n)u_n=0 \text{ and } \lim_{n\to\infty}\int_{\R}F(u_n)=0.
\end{equation}
We thus have, since $\lim_{n\to\infty}\langle \mathcal{J}'(u_n),u_n\rangle=0$, that
\begin{equation*}\label{eqphigoeszero}
\|u_n\|^2=\int_{\mathbb{R}^N}(I_\alpha*|u_n|^{\frac{\alpha}{N}+1})|u_n|^{\frac{\alpha}{N}+1}+o_n(1).
\end{equation*}
On the other hand, $\mathcal{J}(u_n)\to c>0$ as $n\to\infty$, which together with \eqref{vanishing} and the Hardy--Littlewood--Sobolev inequality \eqref{eqhls}, implies that
$
\varliminf_{n\to\infty}\|u_n\|>0.
$
We thus deduce from Lemma \ref{lemcompactness} that
\begin{equation*}
c \geq \varliminf\limits_{n\to\infty}\mathcal{J}(u_n)=\varliminf_{n\to\infty}\Phi(u_n)
- \lim\limits_{n\to\infty}\int_{\R}F(u_n)\geq c_*,
\end{equation*}
a contradiction.

By the P.-L. Lions inequality \citelist{\cite{W}*{Lemma 1.21}\cite{Lions1984CC2}*{lemma I.1}\cite{VanSchaftingen2014}*{(2.4)}}
\[
\int_{\R} \abs{u_n}^q \le C \Bigl(\int_{\R} \abs{\nabla u_n}^2 + \abs{u_n}^2 \Bigr)\Bigl(\sup_{y\in\R}\int_{B_1(y)}\abs{u_n}^q\Bigr)^{1-\frac{2}{q}},
\]
we deduce that there exists a sequence of points $(y_n)_{n\in\mathbb{N}}$ in $\R$ such that
\[
 \varliminf_{n \to \infty} \int_{B_1(y_n)}\abs{u_n}^q>0.
\]
Since the functional $\mathcal{J}$ is invariant under translations, we then define $\tilde{u}_n :=u_n(\cdot+y_n)$, the sequence $(\tilde{u}_n)_{n\in\mathbb{N}}\subset H^1(\R)$ is a bounded $(PS)_{c}$ sequence with converging weakly to some function \(u \in H^1 (\R)\setminus \{0\}\).

Next, we show that $\mathcal{J}'(u)=0$.
Since the sequence $(u_n)_{n\in\mathbb{N}}$ converges weakly to $u$ in $H^1(\R)$, by the Sobolev--Rellich embedding theorem, it converges strongly to $u$ in $L_{\mathrm{loc}}^2(\R)$ and still, up to a subsequence, it converges to $u$ almost everywhere in $\R$.
Note that the sequence $(u_n)_{n\in\mathbb{N}}$ is bounded in $L^2(\R)$, the sequenc $(\abs{u_n}^{\frac{\alpha}{N}+1})_{n\in\mathbb{N}}$ is therefore bounded and converges weakly to $\abs{u}^{\frac{\alpha}{N}+1}$ in $L^{\frac{2N}{N+\alpha}}(\R)$ \cite{Willem}*{Proposition 5.4.7}.
Since the Riesz potential is a linear bounded map from $L^{\frac{2N}{N+\alpha}}$ to $L^{\frac{2N}{N-\alpha}}(\R)$, the sequence  $(I_\alpha*\abs{u_n}^{\frac{\alpha}{N}+1})_{n \in \mathbb{N}}$ converges weakly to $I_\alpha*\abs{u}^{\frac{\alpha}{N}+1}$ in $L^{\frac{2N}{N-\alpha}}(\R)$.
On the other hand, by the Sobolev--Rellich  embedding theorem again, the sequence $(\abs{u_n}^{\frac{\alpha}{N}})_{n\in\mathbb{N}}$ converges strongly to $\abs{u}^{\frac{\alpha}{N}}$ in $L_{\mathrm{loc}}^{\frac{2N}{\alpha}}(\R)$ and
it follows that for any $\varphi\in C^{\infty}_c(\R)$, as $n\to\infty$,
$$
\int_{\R}(I_\alpha*|u_n|^{\frac{\alpha}{N}+1})|u_n|^{\frac{\alpha}{N}-1}u_n\varphi\to\int_{\R}(I_\alpha*|u|^{\frac{\alpha}{N}+1})|u|^{\frac{\alpha}{N}-1}u\varphi.
$$
Similarly, by the assumptions $(f_1)$ and $(f_2)$ on $f$, we have that for any $\varphi\in C^{\infty}_c(\R)$, as $n\to\infty$,
 $$
 \int_{\R}f(u_n)\varphi\to\int_{\R}f(u)\varphi,
 $$
which, together with the fact that the smooth test function set $C^{\infty}_c(\R)$ is dense in $H^1(\R)$ gives that $\mathcal{J}'(u)=0$.

By taking $\bar{\mu}=\min\{\frac{2\alpha}{N}+2,\mu\}>2$, on the one hand, by Fatou's lemma, we see that
\begin{equation}
\label{eqenergyinequality2}
\begin{split}
\mathcal{J}(u)&=\mathcal{J}(u)-\frac{1}{\bar{\mu}}\langle\mathcal{J}'(u),u\rangle\\
&=\Big(\frac{1}{2}-\frac{1}{\bar{\mu}}\Big)\|u\|^2+\Big(\frac{1}{\bar{\mu}}-\frac{N}{2(N+\alpha)}\Big)
\int_{\R}(I_\alpha*|u|^{\frac{\alpha}{N}+1})|u|^{\frac{\alpha}{N}+1}\\
&\qquad +\int_{\R}\frac{1}{\bar{\mu}}f(u)u-F(u)\\
&\leq \Big(\frac{1}{2}-\frac{1}{\bar{\mu}}\Big)\|u_n\|^2+\Big(\frac{1}{\bar{\mu}}-\frac{N}{2(N+\alpha)}\Big)\int_{\R}
(I_\alpha*|u_n|^{\frac{\alpha}{N}+1})|u_n|^{\frac{\alpha}{N}+1}\\
&\qquad+\int_{\R}\frac{1}{\bar{\mu}}f(u_n)u_n-F(u_n)+o_n(1)\\
&=\mathcal{J}(u_n)-\frac{1}{\bar{\mu}}\langle\mathcal{J}'(u_n),u_n\rangle+o_n(1)\to c.
\end{split}
\end{equation}
On the other hand, we have
\begin{equation}
\label{eqenergyinequality1}
\mathcal{J}(u)= \mathcal{J}(u)-\frac{1}{\bar{\mu}}\langle\mathcal{J}'(u),u\rangle=\Big(\frac{1}{2}-\frac{1}{\bar{\mu}}\Big)\|u\|^2>0.
\end{equation}
Then the lemma follows.
\end{proof}
Now we are in position to prove Theorem \ref{thm1.1}.
\begin{proof}[Proof of Theorem~\ref{thm1.1}]
By Proposition~\ref{propMountainpass}, there exits a Palais--Smale sequence $(u_n)_{n\in\mathbb{N}}$ by the mountain pass theorem (see for example \cite{AmbrosettiRabinowitz1973,Rabinowitz1986,Struwe,W}) at the energy level $c_0$ defined by \eqref{eqMountainpassenergy}, it then follows from Lemma~\ref{lemEnergyestimates} that $c_0\in(0, c_*)$.
The sequence $(u_n)_{n\in\mathbb{N}}$ is bounded in $H^1(\R)$: indeed, by taking $\bar{\mu}=\min\{\frac{2\alpha}{N}+2,\mu\}>2$ and by the assumptions on $f$, we see that
\begin{equation}\label{boundnesssequence}
\begin{split}
c_0+\|u_n\|&\geq \mathcal{J}(u_n)-\frac{1}{\bar{\mu}}\langle \mathcal{J}'(u_n),u_n\rangle\\
&=\Big(\frac{1}{2}-\frac{1}{\bar{\mu}}\Big)\|u_n\|^2+\Big(\frac{1}{\bar{\mu}}-\frac{N}{2(N+\alpha)}\Big)
\int_{\mathbb{R}^N}(I_\alpha*|u_n|^{\frac{\alpha}{N}+1})|u_n|^{\frac{\alpha}{N}+1}\\
&\quad +\int_{\R}\frac{1}{\bar{\mu}}f(u_n)u_n-F(u_n)\geq \Big(\frac{1}{2}-\frac{1}{\bar{\mu}}\Big)\|u_n\|^2.
\end{split}
\end{equation}
It follows that, up to a subsequence,  $u_n\rightharpoonup u$ weakly in $H^1(\R)$, by the classical Sobolev--Rellich embedding theorem, we see that  $u_n\to u$ strongly in $L_{\mathrm{loc}}^{q}(\R)$, and $u_n\to u$ almost everywhere in $\R$. Then Lemma \ref{lemnontrivialsolution} infers that $u$ is nontrivial critical point of the functional $\mathcal{J}$ and $\mathcal{J}(u)\in(0,c_0]$.

In what follows, we conclude the proof of Theorem~\ref{thm1.1} by showing that the minimization problem defined by \eqref{minimizationproblem} has a minimizer. Let $(v_n)_{n\in\mathbb{N}}$ be a sequence of nontrivial solutions to \eqref{eqcs} such that
 $\lim_{n\to\infty}\mathcal{J}(v_n)=m_0.$
We first observe that $m_0\leq c_0<c_*$. Since $\mathcal{J}'(v_n)=0$, by taking  $\bar{\mu}$ as before, we have
\begin{equation}
\label{eqBoundedabove}
m_0+o_n(1)=\mathcal{J}(v_n)=\mathcal{J}(v_n)-\frac{1}{\bar{\mu}}\langle \mathcal{J}'(v_n),v_n\rangle\geq \Big(\frac{1}{2}-\frac{1}{\bar{\mu}}\Big)\|v_n\|^2,
\end{equation}
which means that the sequence $(v_n)_{n\in\mathbb{N}}$ is bounded in $H^1(\R)$.
By combining with  $(f_1)$ and $(f_2)$, we have that
\begin{equation}
\label{lowerbound}
\begin{split}
\|v_n\|^2&=\int_{\R}(I_\alpha*\abs{v_n}^{\frac{\alpha}{N}+1})\abs{v_n}^{\frac{\alpha}{N}+1}+\int_{\R}f(v_n)v_n\\
&\leq C_H\Big(\int_{\R}\abs{v_n}^2\Big)^{\frac{\alpha}{N}+1}+\frac{1}{2}\int_{\R}\abs{v_n}^2+C\int_{\R}\abs{v_n}^q\\
&\leq C\|v_n\|^{{\frac{2\alpha}{N}+2}}+\frac{1}{2}\|v_n\|^2+C\|v_n\|^q.
\end{split}
\end{equation}
It then follows that
$
\varliminf_{n\to\infty}\|v_n\|>0,
$
which, together with \eqref{eqBoundedabove} implies that $m_0>0$.
Since the sequence $(v_n)_{n\in\mathbb{N}}$ is a bounded $(PS)_{m_0}$ sequence for the functional $\mathcal{J}$, we deduce from Lemma \ref{lemnontrivialsolution} that $v_n\rightharpoonup v\neq 0$ weakly in $H^1(\R)$ and
$$
\mathcal{J}'(v)=0 \text{ and } \mathcal{J}(v)\in(0, m_0].
$$
On the other hand, by the definition of $m_0$, we conclude that
$\mathcal{J}(v)=m_0$.
Hence, $v$ is a ground state solution of \eqref{eqcs}.
\end{proof}
\section{Schwarz symmetrization method}
\label{sectionsymmetricvarprin}
 In this section, we introduce an alternative proof, based on symmetric minimax principle \cite{JeanVSCCM}*{Theorem 3.2}, to show that if nonlinear perturbation $f$ has the additional symmetric property to be odd and to have constant sign on $(0,+\infty)$, the functional $\mathcal{J}$ has a nontrivial radially symmetric solution when $N\geq 2$, which turns out to be  a ground state solution.

 We first recall some elements of the theory of polarization of functions \citelist{\cite{Baernstein1994}\cite{BrockSolynin2000}}. Assume that $H\subset \R$ is a closed half-space and that $\sigma_H$ is the reflection with respect to $\partial H$. The polarization $u^H:\R\to\mathbb{R}$ of $u:\R\to\mathbb{R}$ is defined for $x\in\mathbb{R}^N$ by
$$u^H(x)=
\begin{cases}
\max\{u(x),u(\sigma_H(x))\} &\text{ if } x\in H;\\
\min\{u(x),u(\sigma_H(x))\} &\text{ if } x\not\in H.
\end{cases}
$$
The following properties of polarization are of use in our arguments, see for example \cite{MVTAMS}*{Lemmas 5.4 and 5.5} and \cite{JeanVSCCM}*{Proposition 2.3}.
\begin{lemma}[Polarization inequality]
\label{polarzationinequality} Let $\alpha\in(0,N)$ and $H\subset\R$ be a closed half-space. If $u\in H^1(\R)$, then $u^H\in H^1(\R)$,
$$\int_{\R}\abs{\nabla u^H}^2=\int_{\R}\abs{\nabla u}^2, $$
 and for any $p\geq \frac{\alpha}{N}+1$ with $\frac{1}{p}\geq \frac{N-2}{N+\alpha}$,
 $$ \int_{\R}(I_\alpha*\abs{u}^p)\abs{u}^p\leq \int_{\R}(I_\alpha*\big|\abs{u}^H\big|^p)\big|\abs{u}^H\big|^p. $$
Moreover, there holds $$\int_{\R}\varphi(u^H)=\int_{\R}\varphi(u)\text{ for any } \varphi\in C(\mathbb{R};[0,+\infty)).$$
\end{lemma}
We now recall the Schwarz symmetrization. We say a Lebesgue measurable function $u:\mathbb{R}^N\rightarrow\mathbb{R}$ is vanish at infinity if $\mathcal{L}^N\left(\{x|~|u(x)|>t\}\right)$ is finite for all $t>0$, where $\mathcal{L}^N(A)$ is the Lebesgue measure of the measurable subset $A\subset\R$. For a nonnegative function $u$ vanishing at infinity, we recall that the Schwarz symmetrization $u^*$ as a radially-decreasing function from $\mathbb{R}^N$ to $\mathbb{R}$, which has the property that for any $t>0$, $\mathcal{L}^N\left(\{x|~|u^*(x)|>t\}\right)=\mathcal{L}^N \left(\{x|~|u(x)|>t\}\right)$.

Some preliminary knowledge are summarized here on the Schwarz symmetrization for subsequent use, we refer the reader to \cite{LL}*{Theorem 3.7, Lemma 8.17}, \cite{JeanVSCCM}*{Proposition 2.1} for references.
\begin{proposition}
\label{rearrangementinequality}
For any  $u\in H^1({\mathbb{R}^N})$, if $u$ is nonnegative, then  $u^*\in H^1(\mathbb{R}^N)$, and
\begin{equation}\label{PolyaSzego}
\int_{\R}|\nabla u^*|^2\leq\int_{\R}|\nabla u|^2;
\end{equation}
for any $s\in [1,\infty)$, if $u\in L^s(\R)$ is nonnegative, then $u^*\in L^s(\R)$ and
\begin{equation}\label{squarenorm}
\int_{\R}|u^*|^s=\int_{\R}|u|^s.
\end{equation}
\end{proposition}

\begin{proof}[Proof of Theorem \ref{thmsymmetricsolution}]
First we observe that $\mathcal{J}(u)=\mathcal{J}(|u|)$, since $\mathcal{J}$ is an even functional. It then follows from Lemma \ref{polarzationinequality} that, for any closed half-space $H\subset \R$, $$\mathcal{J}(|u|^H)\leq \mathcal{J}(u).$$
Therefore, we can find an almost symmetric Palais--Smale sequence $(u_n)_{n\in\mathbb{N}}$
by the symmetric variational principle \cite{JeanVSCCM}*{Theorem 3.2} with minor modifications at the energy level $c_0$ defined by \eqref{eqMountainpassenergy}, then Lemma~\ref{lemEnergyestimates} ensures that $c_0\in(0, c_*)$. The almost symmetric means that the $(PS)_{c_0}$ sequence $(u_n)_{n\in\mathbb{N}}$ additionally satisfies that, as $n\to\infty$,
\begin{equation}\label{eqalmostsymmetric}
u_n-\abs{u_n}^*\to 0 \text{ strongly in } L^2(\R)\cap L^q(\R).
\end{equation}
Here, $*$ denotes the Schwarz symmetrization  and we omit the verification of the abstract theorem and refer the reader to \cite{JeanVSCCM}*{Theorem 3.2 and Example 2.2 and Theorem 4.5}. Similar as the inequality \eqref{boundnesssequence}, a standard argument shows that $(u_n)_{n\in\mathbb{N}}$ is bounded in $H^1(\R)$, then, up to a subsequence,  $u_n\rightharpoonup u$ weakly in $H^1(\R)$ as $n\to\infty$, by the classical Sobolev--Rellich embedding theorem, we see that  $u_n\to u$ strongly in $L_{\mathrm{loc}}^{q}(\R)$, and $u_n\to u$ almost everywhere in $\R$.

We claim that $u\neq 0$ is a radial symmetric solution of \eqref{eqcs}. Let us consider the sequence $(|u_n|^*)_{n\in\mathbb{N}}$, the Schwarz symmetrization of $(|u_n|)_{n\in\mathbb{N}}$. It is also bounded in $H^1(\R)$ by the P\'{o}lya--Szeg\H{o} inequality \eqref{PolyaSzego} and by the Cavalieri principle \eqref{squarenorm}. Since $N\geq 2$, by the Strauss' compactness embedding theorem \citelist{\cite{Straus1977}\cite{W}*{Corollary 1.26}}, the sequence $(\abs{u_n}^*)_{n\in\mathbb{N}}$ is compact in $L^q(\mathbb{R}^N)$, then we see that $\abs{u_n}^*\to u$ strongly in $L^q(\R)$.
Suppose that $u=0$, then it follows from \eqref{eqalmostsymmetric} that $u_n\to 0$ strongly in $L^q(\R)$. As in the proof of Lemma \ref{lemnontrivialsolution}, we see that
$c_0\geq c_*$, which is impossible. Hence, $u\neq 0$ and $\mathcal{J}'(u)=0$. Similarly as estimates \eqref{eqenergyinequality2} and \eqref{eqenergyinequality1}, we deduce that $\mathcal{J}(u)\in(0,c_0]$. The symmetry of $u$ follows from that, up to a subsequence, $|u_n|^*$ converges to $u$ strongly in $L^q(\R)$.

Let $v\in H^1(\R)\setminus\{0\}$ be another solution of \eqref{eqcs} such that $\mathcal{J}(v)\leq \mathcal{J}(u)$, following an argument of L. Jeanjean and H. Tanaka \cite{Jeanjeanremark}, we construct a special path $\gamma_0\in C([0,1]; H^1(\R))$ such that $\gamma_0\in\Gamma$, defined by \eqref{insectionset}, and
$$\gamma_0(1/2)=v,  \;\text{ and for each }\; t\in[0,1]\setminus\{1/2\},\; \mathcal{J}(\gamma(t))<\mathcal{J}(v).$$
Therefrom, by the definition of $c_0$, we have that $$c_0\leq\max_{t\geq 0}\mathcal{J}(\gamma_0(t))\leq\mathcal{J}(v)\leq \mathcal{J}(u)\leq c_0.$$
This  means that $u$ is a ground state solution of \eqref{eqcs} and $\mathcal{J}(u)=c_0=m_0$.

In fact, since $v$ is a solution of $\eqref{eqcs}$, it satisfies the following Poho\v{z}aev identity (see for example, \cite{MVTAMS,MVJFA}),
\begin{equation}\label{pohozaevequality}
\frac{N-2}{2}\int_{\R}\abs{\nabla v}^2+\frac{N}{2}\int_{\R}\abs{v}^2=\frac{N}{2}\int_{\R}(I_\alpha*\abs{v}^{\frac{\alpha}{N}+1})\abs{v}^{\frac{\alpha}{N}+1}+N\int_{\R}F(v).
\end{equation}
We define $\tilde{\gamma}:[0,+\infty)\to H^1(\R)$ by
$$\tilde{\gamma}(\tau)(x)=\begin{cases}
&v(\frac{x}{\tau})\qquad\text{ if } \tau>0;\\
&0 \qquad \quad\,\,\,\text{ if } \tau=0.
\end{cases}
$$
Observe that for every $\tau>0$,
$$\int_{\R}\abs{\nabla \tilde{\gamma}(\tau)}^2+\abs{\tilde{\gamma}(\tau)}^2=\tau^{N-2}\int_{\R}\abs{\nabla v}^2+\tau^N\int_{\R}\abs{v}^2,$$
we see that $\tilde{\gamma}$ is continuous at $0$ only if $N\geq 3$, in the case that $N=2$, we modify the path
$\tilde{\gamma}:[0,+\infty)\to H^1(\mathbb{R}^2)$ by
$$\tilde{\gamma}(\tau)(x)=\begin{cases}
&v(\frac{x}{\tau})\qquad\qquad\text{ if } \tau\geq \tau_0;\\
&\frac{\tau}{\tau_0}v(\frac{x}{\tau_0}) \qquad\,\,\,\,\,\text{ if } \tau\leq\tau_0,
\end{cases}
$$
for some suitable sufficiently small $\tau_0<1$. Plugging \eqref{pohozaevequality} into the functional $\mathcal{J}$, we have that
\begin{equation}\label{energywithPohozaevIdentityN=2}
\begin{split}
&\mathcal{J}(\tilde{\gamma}(\tau))\\&=\frac{\tau^{N-2}}{2}\int_{\R}\abs{\nabla v}^2+\frac{\tau^N}{2}\int_{\R}\abs{v}^2-\frac{N\tau^{N+\alpha}}{2(N+\alpha)}\int_{\R}(I_\alpha*\abs{v}^{\frac{\alpha}{N}+1})\abs{v}^{\frac{\alpha}{N}+1}-\tau^{N}\int_{\R}F(v)\\
&=\Big(\frac{\tau^{N-2}}{2}-\frac{(N-2)\tau^N}{2N}\Big)\int_{\R}\abs{\nabla v}^2+\Big(\frac{\tau^N}{2}-\frac{N\tau^{N+\alpha}}{2(N+\alpha)}\Big)\int_{\R}(I_\alpha*\abs{v}^p)\abs{v}^p.
\end{split}
\end{equation}
It can be checked directly that, $\mathcal{J}\circ\tilde{\gamma}$ achieves the unique strict global maximum at $\tau=1$ for $N\geq 3$, and also have a unique maximum point at $\tau=1$ for $\tau\geq \tau_0$ when $N=2$. We shall choose a suitable $\tau_0<1$ small enough such that $\mathcal{J}\circ\tilde{\gamma}$ is strictly increasing when $\tau\leq \tau_0$ for $N=2$, it then follows that $\tau=1$ is still the unique strictly global maximum point of $\mathcal{J}\circ\tilde{\gamma}$.
Therefore, the desired path $\gamma_0$ can be defined by a suitable change of variable since $\lim_{\tau\to+\infty}\mathcal{J}(\tilde{\gamma}(\tau))=-\infty$.
We finally close our proof by choosing such an appropriate $\tau_0<1$ such that $\frac{\dif }{\dif \tau}\mathcal{J}(\tilde{\tau})>0$ for $\tau\leq \tau_0$. Actually, we have
\begin{equation}\label{energyincreasing}
\begin{split}
&\frac{d}{\dif \tau}\mathcal{J}(\tilde{\gamma}(\tau))\\
&=\frac{d}{\dif \tau}\Big(\frac{\tau^2}{2\tau_0^2}\int_{\mathbb R^2}\abs{\nabla v}^2+\frac{\tau^2}{2}\int_{\mathbb R^2}\abs{v}^2-\frac{\tau^{2+\alpha}}{2+\alpha}\int_{\mathbb R^2}(I_\alpha*\abs{v}^{\frac{\alpha}{2}+1})\abs{v}^{\frac{\alpha}{2}+1}-\tau_0^2\int_{\mathbb R^2} F\Big(\frac{\tau}{\tau_0}v\Big)\bigg)\\
&=\tau\bigg(\frac{1}{\tau_0^2}\int_{\mathbb R^2}\abs{\nabla v}^2+\int_{\mathbb R^2}\abs{v}^2-\tau^{\alpha}\int_{\mathbb R^2}(I_\alpha*\abs{v}^{\frac{\alpha}{2}+1})\abs{v}^{\frac{\alpha}{2}+1}-\frac{\tau_0}{\tau} \int_{\mathbb R^2} f\Big(\frac{\tau}{\tau_0}v\Big)v \bigg).
\end{split}
\end{equation}
We deduce from $(f_2)$ that
$$
\frac{\tau_0}{\tau} \int_{\mathbb R^2} \Bigabs{f\Big(\frac{\tau}{\tau_0}v\Big)v}\leq C\bigg(\int_{\mathbb R^2}\abs{v}^2+\Bigabs{\frac{\tau}{\tau_0}}^{q-2}\int_{\mathbb R^2} \abs{v}^q\bigg)\leq C\int_{\mathbb R^2}\abs{v}^2+\abs{v}^q\leq C
$$
Inserting this into \eqref{energyincreasing}, and combining with \eqref{pohozaevequality} that
$$\int_{\mathbb R^2}\abs{v}^2=\int_{\mathbb R^2}(I_\alpha*\abs{v}^{\frac{\alpha}{N}+1})\abs{v}^{\frac{\alpha}{N}+1}+2\int_{\mathbb R^2}F(v)>\int_{\mathbb R^2}(I_\alpha*\abs{v}^{\frac{\alpha}{N}+1})\abs{v}^{\frac{\alpha}{N}+1},$$
we thus deduce that there exists $\tau_0\in(0,1)$ small enough such that for any $\tau\leq \tau_0$,
$$
\frac{d}{\dif \tau}\mathcal{J}\bigl(\tilde{\gamma}(\tau)\bigr)\geq \tau\bigg(\frac{1}{2\tau_0^2}\int_{\mathbb R^2}\abs{\nabla v}^2\bigg)>0,
$$
which completes the proof.
\end{proof}

\section{Proof of Theroem~\ref{thm1.3}}
\label{finalsection}
In the process of finding ground state solutions of problem \eqref{eqck}, the following limit problem plays a significant role.
 \begin{equation*}
 \tag{$\mathcal{C}_\infty$}
 \label{eqcinfty}
\left\{
\begin{aligned}
&-{\Delta}u+u=\big(I_\alpha*|u|^{\frac{\alpha}{N}+1}\big)|u|^{\frac{\alpha}{N}-1}u+ K_\infty |u|^{q-2}u \qquad \text{ in } \R\\
&u\in H^1{(\R}),
\end{aligned} \right.
\end{equation*}
the associated functional is defined as
$$
\mathcal{K}_{\infty}(u)=\frac{1}{2}\int_{\R}|\nabla u|^2+|u|^2-\frac{N}{2(N+\alpha)}\int_{\R}(I_\alpha*|u|^{\frac{\alpha}{N}+1})|u|^{\frac{\alpha}{N}+1}-\frac{K_\infty}{q}\int_{\R}|u|^q.
$$
To alleviate the notation, we define
$$
c_\infty=\inf\,\big\{\mathcal{K}_\infty(u)\;|\;u\in H^1(\R)\setminus\{0\}\,\,\text{ and }\,\, \mathcal{K}_\infty'(u)=0\big\}.
$$

\begin{proof} [Proof of Theorem~\ref{thm1.3}]
Let
$$
\mathcal{K}(u)=\frac{1}{2}\|u\|^2-\frac{N}{2(N+\alpha)}
\int_{\R}(I_\alpha*|u|^{\frac{\alpha}{N}+1})|u|^{\frac{\alpha}{N}+1}-\frac{1}{q}\int_{\R}K|u|^q,
$$
then  critical points of the functional $\mathcal{K}$ are weak solutions of \eqref{eqck}, and vice versa.
Similar as Proposition~\ref{propMountainpass}, we see that the functional $\mathcal{K}$ also has the mountain pass geometry, we then have a minimax description at $c_K$, defined by
\begin{equation*}
\label{eqMountainpassenergyeqck}
c_K=\inf_{\gamma\in \Gamma_K}\max_{t\in[0,1]}\mathcal{K}(\gamma(t)),
\end{equation*}
where
$$\Gamma_K=\{\gamma\in C([0,1], H^1(\R))|\gamma(0)=0,\,\,\,\mathcal{K}(\gamma(1))<0\}.$$
By the assumptions on $K$, we infer that $c_K< c_\infty$. In fact, by Theorem \ref{thm1.1}, the level $c_\infty$ is attained at a ground state solution $u_\infty\in H^1(\R)$ of the limit problem \eqref{eqcinfty}. We can assume without loss of generality that there exists a set of positive measure on which $K>K_\infty$ --- otherwise Theorem~\ref{thm1.3} is a special case of Theorem~\ref{thm1.1}. We then deduce that $\mathcal{K}(tu_\infty)<\mathcal{K}_\infty(tu_\infty)$ for all $t>0$, from this we have that
 \begin{equation*}\label{eqenergycomparision}
 c_K\leq \max_{t\geq 0} \mathcal{K}(tu_\infty)=\mathcal{K}(t_*u_\infty)< \mathcal{K}_\infty(t_*u_\infty)\leq \max_{t\geq 0}\mathcal{K}_\infty(tu_\infty)=c_\infty<c_*,
 \end{equation*}
 where $t_*>0$ is unique and satisfies that $\langle \mathcal{K}'(t_*u_\infty),t_*u_\infty\rangle=0$. Here we have taken advantage of the monotonicity of the perturbation, thus the Nehari manifold method works, for the detail proofs, we refer to \cite{SW,W}.

 Let $(u_n)_{n\in\mathbb{N}}$ be a $(PS)_{c_K}$ sequence for the functional $\mathcal{K}$, a standard argument shows that $(u_n)_{n\in\mathbb{N}}$ is bounded in $H^1(\R)$. Up to a subsequence, $u_n\rightharpoonup u$ weakly in $H^1(\R)$ as $n\to\infty$, and $u_n$ converges to $u$ almost everywhere in $\R$. By a similar argument as in the proof of Lemma \ref{lemnontrivialsolution}, we see that
 there exist $(y_n)_{n\in\mathbb{N}}\subset\R$ and $\delta>0$
 such that
 $$
 \varliminf_{n\to\infty}\int_{B_1(y_n)}|u_n|^2\geq\delta.
 $$

We next claim that $(y_n)_{n\in\mathbb{N}}$ is bounded in $\mathbb{R}^N$, thus $u$ is a nontrivial solution of \eqref{eqck}.
It then follows from \eqref{eqenergyinequality2} and \eqref{eqenergyinequality1} similarly that $\mathcal{K}(u)\in(0,c_K]$. We now complete this claim indirectly. Suppose that for a subsequence $|y_n|\to+\infty$ as $n\to\infty$,  we define $v_n=u_n(\cdot+y_n)$ and then $(v_n)_{n\in\mathbb{N}}$ is bounded in $H^1(\R)$, and $v_n\rightharpoonup v\neq 0$.
The assumption on the asymptotic shape of the potential
$K$ implies that $v$ is a critical point of \eqref{eqcinfty}. In fact, we first have, for any $w\in H^1(\R)$ as $n\to \infty$, that
\begin{equation*}
\begin{split}
\Big|\int_{\mathbb{R}^N}(K(x)-K_{\infty})&|u_n(x)|^{q-2}u_n(x)w(x-y_n)\dif x\Big|\\
&\leq\int_{B_{|y_n|/2}}(K(x)-K_{\infty})|u_n(x)|^{q-1} |w(x-y_n)|\dif x\\
&\qquad
 +\int_{\mathbb{R}^N\backslash {B_{|y_n|/2}}}(K(x)-K_{\infty})|u_n(x)|^{q-1} |w(x-y_n)|\dif x\\
&\leq 2\|K\|_{L^\infty}\|u_n\|^{q-1}_{L^q}\|w\|_{L^q(\mathbb{R}^N\backslash {B_{|y_n|/2}})}\\
&\qquad+\|K-K_\infty\|_{L^\infty(\mathbb{R}^N\setminus B_{|y_n|/2})}\|u_n\|^{q-1}_{L^q}\|w\|_{L^q}\to 0,
\end{split}
\end{equation*}
since \(B_{\abs{y_n}/2}(-y_n) \subset \R \setminus B_{\abs{y_n}/2}\).
Thus, since $v_n\rightharpoonup v $ weakly in $H^1(\R)$ as \(n \to \infty\), we have that, for any $w\in H^1(\R)$,
\begin{equation*}
\begin{split}
\langle \mathcal{K}_\infty'(v),w \rangle
&=\int_{\mathbb{R}^N}\nabla v\cdot \nabla w+ v w - K_\infty|v|^{q-2}vw -\int_{\R}(I_\alpha*|v|^{\frac{\alpha}{N}+1})|v|^{\frac{\alpha}{N}-1}v w \\
&=\int_{\mathbb{R}^N}\nabla v_n\cdot \nabla w+  v_n w-K_\infty|v_n|^{q-2}v_nw\\
&\qquad -\int_{\R}(I_\alpha*|v_n|^{\frac{\alpha}{N}+1})|v_n|^{\frac{\alpha}{N}-1}v_n w +o_n(1)\\
&=\langle \mathcal{K}'(u_n), w(x-y_n)\rangle +\int_{\mathbb{R}^N}(K(x)-K_{\infty})|u_n|^{p-2}u_n w(x-y_n)\dif x+o_n(1)\\
&=\langle\mathcal{K}'(u_n),w(x-y_n)\rangle +o_n(1)\to 0.
\end{split}
\end{equation*}
We thus deduce that $\mathcal{K}_\infty'(v)=0$. However, Fatou's lemma implies that
\begin{equation*}
\begin{split}
c_K+&o_n(1)\|u_n\|=\mathcal{K}(u_n)-\frac{1}{2}\langle \mathcal{K}'(u_n),u_n\rangle \\
&\geq\frac{\alpha}{2(N+\alpha)}\int_{\R}(I_\alpha*|v|^{\frac{\alpha}{N}+1})|v|^{\frac{\alpha}{N}+1} + \Bigl(\frac{1}{2}-\frac{1}{q}\Bigr)\int_{\mathbb{R^N}}K_\infty|v|^q+o_n(1)\\
&=\mathcal{K}_{\infty}(v)-\frac{1}{2}\langle \mathcal{K}_{\infty}'(v),v\rangle +o_n(1)\geq c_\infty+o_n(1),
\end{split}
\end{equation*}
which contradicts with  $c_K<c_\infty$.

Finally, following the strategy of Theorem \ref{thm1.1}, we consider the minimization problem
 $$
 m=\inf\,\bigl\{\mathcal{K}(v)\;|\; v\in H^1(\R)\setminus\{0\}\text{ and } \mathcal{K}'(v)=0\bigr\}.
 $$
Let $(v_n)_{n\in\mathbb{N}}$ be a minimizing sequence for $m$, we first deduce similarly as \eqref{eqBoundedabove} and \eqref{lowerbound} that  $0<m\leq c_K $, and $(v_n)_{n\in\mathbb{N}}$ is bounded in $H^1(\R)$. Repeating the argument as above,  $v_n\rightharpoonup v\neq 0$ weakly in $H^1(\R)$ and $\mathcal{K}'(v)=0$, then it follows from Fatou's lemma  like in \eqref{eqenergyinequality2} that $\mathcal{K}(v)\leq m$. Hence, $v$ is the ground state solution of \eqref{eqck} that we desire.
\end{proof}
\begin{remark}
Checking through our proof, we see that \eqref{eqck} has a ground state solution when $q=2+\frac{4}{N}$ provided that $K_\infty\geq \Lambda_0$,
the constant that obtained in Theorem \ref{thm1.1}.
\end{remark}

\end{document}